\documentclass[12pt,a4paper]{article}
\usepackage{amssymb}
\usepackage{amsmath,cancel}
\usepackage{graphicx,epsfig}
\usepackage{amssymb,amsfonts,amsmath,amsthm,dsfont,wasysym,pifont,stmaryrd,bbm}
\usepackage{epstopdf,yfonts,hyperref}

\newtheorem{Thm}{Theorem}
\newtheorem{Def}[Thm]{Definition}
\newtheorem{Prop}[Thm]{Proposition}
\newtheorem{Lem}[Thm]{Lemma}
\newtheorem{Cor}[Thm]{Corollary}
\newtheorem{Ex}{Example}

\newtheorem{remark}{Remark}
\newcommand{\1}{\mathds{1}}

\renewcommand{\>}{\right\rangle}
\renewcommand{\~}[1]{\widetilde{#1}}
\newcommand{\8}{\infty}
\newcommand{\E}{\mathcal{E}}

\renewcommand{\H}{\mathcal{H}}
\newcommand{\im}{\mathrm{im}}
\newcommand{\Q}{\mathbb{Q}}
\newcommand{\PSL}{\textnormal{PSL}}

\newcommand{\Sum}[1]{\underset{#1}{{\sum} }}
\newcommand{\Z}{\mathbb{Z}}

\newcommand{\Oplus}[1]{\underset{#1}{\oplus}}

\title{The Mayer-Vietoris Sequence for Graphs of Groups, Property (T), and the First $\ell^2$-Betti Number\\
}

\author{Talia FERN\'OS and Alain VALETTE}

\begin{document}

\baselineskip=16pt

\maketitle

\begin{abstract} 
We explore the Mayer-Vietoris sequence developed by Chiswell for the fundamental group of a graph of groups when vertex groups satisfy some vanishing assumption on the first cohomology (e.g. property (T), or vanishing of the first $\ell^2$-Betti number). We characterize the vanishing of first reduced cohomology of unitary representations when vertex stabilizer have property (T). We find necessary and sufficient conditions for the vanishing of the first $\ell^2$-Betti number. We also study the associated Haagerup cocycle and show that it vanishes in first reduced cohomology precisely when the action is elementary. 
\end{abstract}
\vspace{1cm}


\section{Introduction}

\subsection{1-cohomology}

The first cohomology of a locally compact group $G$ with coefficients in a unitary representation $\pi:G \to \mathcal{U}(\mathcal{H})$ is an object whose study encompasses many interesting themes, such as Property (T), the Haagerup property, and other relatives, as well as the first $\ell^2$-Betti number. For a discrete group $G$ acting on a tree, Chiswell introduced a Mayer-Vietoris sequence for the cohomology of $G$, in terms of the cohomologies of the vertex groups and edge groups \cite{Chi}. The class of groups which admit such actions is quite large and includes limit groups and non-finitely generated countable groups, among others.

In this paper, we first extend Chiswell's sequence, in low degrees, to arbitrary topological groups acting on trees. We study these under the assumption that vertex stabilizers satisfy some condition on the vanishing of $H^1$ (property (T) or vanishing of the first $\ell^2$-Betti number), which leads to two major applications that we now discuss. One concerns the first $\ell^2$-Betti number $\beta^1(G)$; the other, the fact that the first cohomology of $\PSL_2\Q_p$ is non-vanishing precisely when the irreducible representation is special (which is a result of Nebbia \cite{Neb}). 

Let $G$ be a topological group acting without inversion on a tree $\cal{T}$, with quotient graph $X=:G\backslash \cal{T}$. Classical by now is the fact that $G$ admits a  graph of groups decomposition $\pi_1({\cal G}, X, T)$, where $T\subset X$ is a maximal tree and ${\cal G}$ represents the local groups (see Section \ref{Prelim} for more details on Bass-Serre theory). Throughout this paper, we will take such actions and decompositions as interchangeable. Also, in order to avoid cumbersome notation, we assume the graph of groups to be reduced (see Definition \ref{reduced} in Section \ref{redgraphs}).

\begin{Thm}\label{L2 Betti}
  Let $X$ be a graph, $T$ a maximal tree in $X$, 
  $({\cal G}, X)$ a reduced graph of discrete groups, and $G = \pi_1({\cal G}, X, T)$. Assume that $\beta^1(G_v)=0$ for every $v\in V$ and $\sum\frac{1}{|G_e|}<\8$. Then $\beta^1(G)=0 $ if and only if $G$ belongs to one of the following cases:
  
\begin{enumerate}
 \item The graph $X$ is a single vertex and $G= G_v$. 
  \item The graph $X$ is a single loop and $G=\Z\ltimes G_v$.  
 \item The graph $X$ is a single edge and there is an exact sequence $$1\to G_e\to G \to \Z/2*\Z/2\to 1.$$ 
  \item Every edge group is infinite.
\end{enumerate}
\end{Thm}

For a locally compact group $G$, we denote the collection of irreducible unitary representations (up to unitary equivalence) of $G$ by $\hat G$. 
The second main result is an elementary proof of the following theorem, the hypotheses of which guarantee that there is a unique special representation of $G$ (see Section \ref{SectLargeGroups} for more details).

\begin{Thm}\label{twoorbits}\cite{Neb} Let $\mathcal{T}$ be a locally finite, bi-regular tree. Let $G$ be a closed subgroup of $Aut(\mathcal{T})$, acting transitively on $\partial\mathcal{T}$ and with two orbits on $V$. If $\pi\in\hat{G}\backslash\{\sigma\}$, where $\sigma$ is the special representation of $G$ on the first $\ell^2$-cohomology on $\mathcal{T}$, then $H^1(G,\pi)=0$; on the other hand $H^1(G,\sigma)$ is one-dimensional.
\end{Thm}

\subsection{Chiswell's Mayer-Vietoris Sequence}

If $G$ is a discrete group acting on a tree without inversion, and $M$ is any $G$-module,
the cohomology $H^*(G,M)$ can be computed by means of a Mayer-Vietoris sequence due to Chiswell \cite{Chi}:

\begin{equation}\label{MayViet}
0\rightarrow M^G \stackrel{\Delta}{\rightarrow} \prod_{v\in V}M^{G_v}\stackrel{\iota}{\rightarrow} \prod_{e\in A}M^{G_e}
\stackrel{\partial}{\rightarrow} H^1(G,M)\stackrel{\Delta}{\rightarrow} \prod_{v\in V}
H^1(G_v,M)\stackrel{\iota}{\rightarrow} \prod_{e\in A}H^1(G_e,M)\rightarrow...
\end{equation}
\vskip.15in

where:
\begin{itemize}
\item $V$ (resp. $E$) is the vertex set (resp. oriented edge set) of $X$, and $A$ is an orientation, i.e. a choice of 
one edge in every pair of two edges with opposite orientation in $E$;
\item $G_v$ (resp. $G_e$) is the vertex group (resp. edge group) attached with $v\in V$ (resp. $e\in E$);
\item for every subgroup $H\subset G$, the sub-module of $H$-fixed points in $M$ is $M^H$;
\end{itemize}
and the maps $\Delta,\iota,\partial$ will be described later.

For a topological group $G$, a $G$-module $M$ is {\it unitary} if $M$ is a Hilbert space on which 
$G$ acts through a strongly continuous unitary representation. In that case, we also consider the {\it reduced} cohomology $\overline{H}^*(G,M)$, i.e.
the quotient of cocycles by the closure of coboundaries, where the closure is taken in the topology of uniform convergence on compact subsets. 

The Delorme-Guichardet Theorem (see \cite{Guic}, or section 2.12 in \cite{BHV}) says that, for a locally compact, second-countable group $G$, Kazhdan's Property (T) is equivalent to $H^1(G,M)
=0$ for every unitary $G$-module $M$. Since then, Shalom \cite{Shalom1} proved that, for $G$ compactly generated, Property (T) is equivalent
to the vanishing of $\overline{H}^1(G,V)$, for every unitary $G$-module $M$. (The relative analogue of Shalom's theorem fails in general \cite{FernosValette}.)

A locally compact group $G$ has Serre's Property (FA) if every continuous, isometric action of $G$ on a tree preserves a vertex or an edge. It is a result by
Watatani \cite{Wat} that Property (T) implies Property (FA) (see also Lemma \ref{watatani} below).

As a consequence, the fundamental group $G$ of a graph of groups, provided it does not coincide with some vertex group, does {\it not} have
Property (T), as $G$ acts without fixed point on the universal cover of the graph of groups (see \cite{Serre}, section 5.4). The Mayer-Vietoris
sequence (\ref{MayViet}) allows in principle to characterize the unitary $G$-modules $M$ for which $H^1(G,M)\neq 0$. One question we
address in this paper is to characterize the unitary $G$-modules $M$ such that $\overline{H}^1(G,M)\neq 0$. It is possible to write down
the analogue of (\ref{MayViet}) in reduced cohomology, but easy examples show that it will not in general be exact. However, when the vertex groups 
have Property (T) (in particular when they are compact, so that our result covers the case of locally compact groups acting properly 
on trees), we can characterize those unitary $G$-modules $M$ with $\overline{H}^1(G,M)\neq 0$:

\begin{Thm}\label{main}
Let $G$ be a locally compact group acting without inversion on a tree, with vertex-stabilizers having Property (T); let $M$ be a unitary
$G$-module.
\begin{enumerate}
\item [i)] $H^1(G,M)=0$ if and only if the map $\iota:\prod_{v\in V}M^{G_v}\rightarrow \prod_{e\in A}M^{G_e}$ is onto;
\item [ii)] $\overline{H}^1(G,M)=0$ if and only if the map $\iota:\prod_{v\in V}M^{G_v}\rightarrow \prod_{e\in A}M^{G_e}$ has dense image (where $\prod_{v\in V}M^{G_v}$ and $\prod_{e\in A}M^{G_e}$ are endowed with the product topology).
\end{enumerate}
\end{Thm}

Of course the first part of Theorem \ref{main} is an immediate consequence of the Mayer-Vietoris sequence (\ref{MayViet}), we record it only to contrast it with the second part.

As semi-direct products by $\mathbb{Z}$ can be viewed as particular HNN-extensions, we may apply Theorem 1 to them. We will prove:

\begin{Cor}\label{semidir} Let $\theta$ be an automorphism of a locally compact group $\Gamma$ with Property (T). 
Let $M$ be a unitary $G$-module, with
$G=:\Gamma\rtimes_{\theta} \mathbb{Z}$. Let $t$ be the generator of $\mathbb{Z}$ such that $tht^{-1}=\theta(h)\;(h\in \Gamma)$.
\begin{enumerate}
\item [i)] $H^1(G,M)=0$ if and only if 1 is not a spectral value of $t|_{M^{\Gamma}}$;
\item [ii)] $\overline{H}^1(G,M)=0$ if and only if 1 is not an eigenvalue of $t|_{M^{\Gamma}}$.
\end{enumerate}
\end{Cor}

This paper is organized as follows. Section \ref{Prelim} gives a direct construction of Chiswell's Mayer-Vietoris sequence, in degrees 0 and 1: this does not seem to appear explicitly in the literature and is needed for our computations. In Section \ref{HaagerupCoc}, we study a specific 1-cocycle $b\in Z^1(G,\ell^2(\mathcal{E}))$, where $\mathcal{E}$ is the set of oriented edges of $\mathcal{T}$: it is the cocycle such that $\|b(g)\|^2= 2d(gx_0,x_0)$, used to prove that a group acting properly on a tree has the Haagerup property; we call $b$ the {\it Haagerup cocycle}. 
We re-prove the known fact that $b$ is trivial in $H^1(G,\ell^2(\mathcal{E}))$ if and only if $G$ has a fixed vertex; our proof provides a cohomological characterization of Serre's Property (FA). We also prove that $b$ is trivial in $\overline{H}^1(G,\ell^2(\mathcal{E}))$ if and only if the $G$-action on $\mathcal{T}$ is elementary, i.e. $G$ has a finite orbit in $\mathcal{T}\cup\partial\mathcal{T}$. Section \ref{PfMainTh} is dedicated to the proof of Theorem \ref{main} and Corollary \ref{semidir}. In Section \ref{SectL2Betti} we study the consequences of the Mayer-Vietoris sequence on the vanishing of the first $\ell^2$-Betti number for discrete groups acting on trees; in particular we prove Theorem \ref{L2 Betti}. Section \ref{SectLargeGroups} studies the connection with ``large" closed groups of automorphisms of a locally finite tree, and it is there that one will find the proof of Nebbia's  Theorem \ref{twoorbits}. 


{\bf Acknowledgements:} The first named author would like to thank Universit\'e de Neuch\^atel for their hospitality during a visit which made this joint work possible. The second named author thanks C. Weibel for pointing out Chiswell's paper, and D. Gaboriau for useful exchanges on $\ell^2$-Betti numbers.

\section{Preliminaries}\label{Prelim}

The natural framework for our study is Bass-Serre Theory (section 5 in \cite{Serre}), of which we first recall the relevant parts.

A {\it graph} is a pair $X=(V,E)$ where $V$ is the set of vertices, $E$ is the set of oriented edges; $E$ is equipped with a fixed-point free
involution $e\mapsto\overline{e}$ and with maps $E\rightarrow V: e\mapsto e_+$ and $E\rightarrow V: e\mapsto e_-$ ($e_-$ is the initial vertex and $e_+$ the
terminal vertex of the edge $e$), such that $\overline{e}_+=e_-$ for every $e\in E$. An {\it orientation} of $X$ is the choice of a fundamental
domain for the involution on $E$.

A {\it graph of groups} $({\cal G}, X)$ is the data of a connected graph $X=(V,E)$ and, for every $v\in V$ a discrete group $G_v$, for every edge
$e\in E$ a discrete group $G_e$ such that $G_e =G_{\overline{e}}$ and a monomorphism $\sigma_e: G_e\rightarrow G_{e_+}$.

Let $T$ be a maximal tree in $X$. Let $F(E)=\langle t_e\;(e\in E)|\varnothing\rangle$ be the free group on $E$.
The {\it fundamental group} $G=:\pi_1({\cal G},X,T)$ is the quotient of the free product $(\ast_{v\in V}G_v)\ast F(E)$ by the following 
set of relations:
\begin{equation}
\left\{
\begin{array}{ll}
t_e\sigma_e(g_e)t_e^{-1}=\sigma_{\overline{e}}(g_e) & (e\in E,\,g_e\in G_e)\\
t_et_{\overline{e}}=1 & (e\in E) \\
t_e=1 & (e\in E(T))\\
 \end{array}\right.
\end{equation}
 where $E(T)$ is the edge set of $T$. Assume that some orientation $A$ has been chosen. For $e\in A$, we shall identify $G_e$ with 
 $\sigma_e(G_e)$, and we shall denote by $\theta_{e}$ (instead of $\sigma_{\overline{e}}$) the monomorphism $G_e\rightarrow G_{e_-}$.
 Then $G$ can also be described (see \cite{Chi}, p. 67) as the quotient of the free product $(\ast_{v\in V}G_v)\ast F(A)$
 by the relations:
 \begin{equation}\label{pres}
 \left\{
\begin{array}{ll}
t_eg_et_e^{-1}=\theta_e(g_e) & (e\in A,\,g_e\in G_e)\\
t_e=1 & (e\in E(T)\cap A)\\
 \end{array}\right.
 \end{equation}
 
Recall from \cite{Serre} that the universal cover $\mathcal{T}$ of $X$  has vertex set  $\~V = \underset{v\in V}{\sqcup} G/G_v$ and oriented edge set $\~A = \underset{e\in A}{\sqcup} G/G_e$. Furthermore, if $\pi: \mathcal{T} \to X$ denotes the canonical projection, then there is a lifting $X\to \mathcal{T}$ denoted by $x\mapsto \~x$ such that $\pi(\~x) = x$, and $\~T$, the lift of the maximal tree $T$ is a subtree of $\mathcal{T}$. With this notation in place, we have that\footnote{We note that our convention here does not agree with Serre's, who identifies edge groups $G_e$ with their image in $G_{e_+}$.}, for $e\in A$:

\begin{equation}\label{extrorig}
(g\~e)_+ = g\~{e_+}\qquad \text{ and }\qquad (g\~e)_- = gt_e^{-1}\~{e_-}.
\end{equation}

This gives rise to the following short exact sequences of $G$-modules

\begin{equation}\label{shortex}
0\to \Z(\~A)\overset{\delta}{\to} \Z(\~V)  \overset{q}{\to} \Z \to 0. 
\end{equation}

As $G$-modules, these decompose as $\Z(\~A) = \Oplus{e\in A} \Z(G.\~e)$ and $\Z(\~V) = \Oplus{v\in V}\Z(G.\~v) $, so that
 $\delta(g.\~e) = g\~e_+ - gt_e^{-1}\~e_-$ is just the boundary operator on the tree $\mathcal{T}$, and $q$ is the augmentation defined by $q(g.\~v) = 1$. 

Now, for two vertices $v,w\in V$, we denote by $[v,w]$, (respectively $[g\~v,g'\~w]$) the unique \emph{oriented} edge path from $v$ to $w$ (respectively from $g\~v$ to $g'\~w$) in the maximal tree $T$, (respectively in $\mathcal{T}$).

 For $e\in E$ or $e\in\~E$, we define

$$\varepsilon_{vw}(e)= \left\{
\begin{array}{ll}
0 & \mbox{if $e\notin [v,w]$;}\\
+1 & \mbox{if $e\in [v,w]$ and $e$ points away from $v$;} \\
-1 & \mbox{if $e\in [v,w]$ and $e$ points towards $v$.}
\end{array}\right. $$

Note that if $e\in E\setminus E(T)$ then $\varepsilon_{vw}(e) = 0$ for every $v,w \in V$. Furthermore,  since $T$ lifts to $\~T$ we have that $\~{[v,w]} =[\~v,\~w]$ and  so  $\pi[\~v,\~w]= \pi\~{[v,w]} = [v,w]$, in particular  $\varepsilon_{\~v\~w}(\~e) = \varepsilon_{vw}(e)$ for every $v,w\in V$ and $e\in E$. 

\begin{remark}\label{epsilonAcocyle}
 Observe that the following hold for all vertices $u,v,w\in V$, or $u,v,w\in \~V$: 
$$\varepsilon_{vw} = - \varepsilon_{wv} $$ $$ \varepsilon_{uv} + \varepsilon_{vw} = \varepsilon_{uw}$$
\end{remark}

This allows us to define $\int_v^w [v,w] = \sum\varepsilon_{vw}(e) e$, where the sum is taken over ${e\in A}$, or ${e\in \~A}$, according to whether $v,w\in V$ or $\~V$. 

Using this, we fix a base vertex $v_0 \in V$ and define $s: \Z(\~V) \to \Z(\~A)$ by

$$s(g\~v) = \int_{\~v_0}^{g\~v}  [\~v_0,g\~v].$$

It is then straightforward to verify that $s$ is a left-inverse to $\delta$:

$$s\circ \delta (g\~e) = \int_{\~v_0}^{(g\~e)_+} [\~v_0,(g\~e)_+] -  \int_{\~v_0}^{(g\~e)_-} [\~v_0,(g\~e)_-] = g\~e.$$

Now, for abelian groups $C,D$ denote by $\hom(C,D)$ the abelian group of all homomorphisms $C\to D$. If furthermore, $C$ and $D$ are $G$-modules, then $\hom(C,D)$ is also a $G$-module by taking $g.f := g\circ f\circ g^{-1}$ and so $\hom(C,D)^G$ is the collection of $G$-equivariant homomorphisms, (i.e. intertwiners) from $C$ to $D$. For $f:C\rightarrow D$ a homomorphism, and $M$ an abelian group, the {\it transposed} homomorphism $f^t:\hom(D,M)\rightarrow \hom(C,M)$ is defined by $f^t(g)=g\circ f$.

Let $M$ be a $G$-module. Transposing the exact sequence (\ref{shortex}), we get a new exact sequence:
\begin{equation}\label{ShortHom}
0 \to M = \hom(\Z,M) \overset{q^t}{\to} \hom(\Z(\~V), M) \overset{\delta^t}{\to} \hom(\Z(\~A) , M) \to 0,
\end{equation}
where $(q^t(m))(g\~v) \equiv m$, and $(\delta^t(f))(g\~e) = f(\delta(g\~e)) = f((g\~e)_+)-f((g\~e)_-)$. Furthermore, we have for $\omega \in \hom(\Z(\~A) , M)$

$$(s^t(\omega))(g\~v) = \int_{\~v_0}^{g\~v} \omega:=\omega(\int_{\~v_0}^{g\~v}[\~v_0, g\~v]).$$

We note that for $\omega=(\omega_e)_{e\in A}\in \prod_{e\in A} M^{G_e}$, and any three vertices $u,v,w\in V$, we then have that 
\begin{eqnarray*}
\int_v^w \omega&=& \int_{\~v}^{\~w} \omega, \\
\int_v^w \omega&=&-\int_w^v \omega, 
\\
\int_u^v \omega + \int_v^w \omega&=&\int_u^w \omega.
\end{eqnarray*}

Furthermore, if $e\in E(T)\cap A$ then 
\begin{equation}\label{edge}
\int_{e_-}^{e_+}\omega=\omega_e.
\end{equation}


Let $C^k(G,M)$ be the space of $k$-cochains on $G$ with coefficients in $M$, i.e. the set of maps $G^k\rightarrow M$. Applying the functor $C^*(G, \cdot)$ to the short exact sequence (\ref{ShortHom}), we get a commutative diagram:
$$\left.\begin{array}{ccccccccc}
 & & \uparrow &  & \uparrow &  & \uparrow & & 
  \\
0 & \rightarrow & C^1(G,M) & \rightarrow & C^1(G,\hom(\mathbb{Z}(\tilde{V}),M))& \rightarrow & C^1(G,\hom(\mathbb{Z}(\tilde{A}),M))& \rightarrow & 0 \\
 & & \uparrow &  & _{\partial_0}\uparrow &  & \uparrow & & 
  \\0 & \rightarrow & M & \rightarrow & \hom(\mathbb{Z}(\tilde{V}),M) & \rightarrow & \hom(\mathbb{Z}(\tilde{A}),M)& \rightarrow & 0\end{array}\right.$$
where  $\partial_0 : \hom(\Z(\~V), M) \to C^1(G,\hom(\Z(\~V), M))$ is the map given by $(\partial _0\phi)(g) = (g-1) \phi$. This in turn  yields
the long exact sequence \cite[Theorem 2]{Chi}:
$$0\to M^G\overset{q^t_*}{\to} \hom(\Z(\~V),M)^G \overset{\delta^t_*}{\to} \hom(\Z(\~A),M)^G \overset{\partial}{\to} H^1(G,M)\to \cdots,$$
where the connecting map is given by the Snake Lemma:

$$\~\partial\omega = ((q_*^t)^{-1}\circ \partial_0 \circ (\delta_*^t)^{-1})(\omega).$$

The decomposition of $\~V$ and  $\~A$ into $G$ orbits yields that $H^k(\Z(\~V), M) = \underset{v\in V}{\prod}H^k(\Z(G/G_v),M)$ and $H^k(\Z(\~A), M) = \underset{e\in A}{\prod}H^k(\Z(G/G_e),M)$. By Shapiro's Lemma \cite[Proposition 6.2]{Brown} we have that for any $x\in V\sqcup A$
$$H^k(G_x, M) \cong H^k(G, \hom(\Z(G/G_x),M)).$$
Hence, recalling that $M^{G_x} \cong \hom(\Z(G/G_x), M)^G$ and setting $\Delta = q_*^t$ and $\iota = \delta_*^t$ we get Chiswell's Mayer-Vietoris sequence:

$$0\to M^G\overset{\Delta}{\to} \underset{v\in V}{\prod}M^{G_v} \overset{\iota}{\to}\underset{e\in A}{\prod}M^{G_e}\overset{\partial}{\to} H^1(G,M)\to \cdots$$

 We proceed to record the maps $\Delta$ and $\iota$ appearing in the Mayer-Vietoris sequence $(\ref{MayViet})$ and observe that these are well defined for $G$ a topological group.
 \begin{itemize}
 \item The map $\Delta: H^k(G,M)\rightarrow\prod_{v\in V} H^k(G_v,M)$ arises by restricting the action from $G$ to the $G_v$'s;
 in particular $\Delta: M^G\rightarrow \prod_{v\in V}M^{G_v}$ is given by
 $(\Delta m)_v = m$ (where $m\in M^G,\,v\in V$).
 \item The map $\iota: \prod_{v\in V}H^k(G_v,M)\rightarrow \prod_{e\in A}H^k(G_e,M)$ is given at the level of $k$-cocycles by
 $$(\iota\omega)_e = \omega_{e_+}|_{G_e}-t_e^{-1}(\omega_{e_-}\circ\theta_e)$$
 where $e\in A,\;\omega\in\prod_{v\in V}Z^k(G_v,M)$. In particular $\iota:\prod_{v\in V}M^{G_v}\rightarrow \prod_{e\in A}M^{G_e}$ is given
 by $(\iota f)_e = f_{e_+}-t_e^{-1}f_{e_-}$, for $f\in \prod_{v\in V}M^{G_v}$.
 \end{itemize}

\begin{Ex} Let $G=\mathbb{F}_2=\langle t_1,t_2|\varnothing\rangle$ be the free group on two generators, viewed as the fundamental
group of the graph of groups with one vertex and two edges, and all groups trivial. For $M$ a $G$-module, the map $\iota: M
\rightarrow M\oplus M$ is given by $f\mapsto ((1-t_1^{-1})f,(1-t_2^{-1})f)$. If $M$ is a unitary $G$-module, it was shown by Guichardet \cite{Guic} that 
$\iota$ is {\it not} onto, so that $H^1(\mathbb{F}_2,M)\neq 0$ for every unitary $G$-module $M$. Note that examples of unitary,
irreducible $G$-modules $M$ with $\overline{H}^1(G,M)=0$ were constructed in \cite{MarVal}.
\end{Ex}

\medskip
Next we consider the connecting map $\partial: \prod_{e\in A}M^{G_e}\rightarrow H^1(G,M)$: it is not described explicitly in Chiswell's paper \cite{Chi}, nor are we aware of any published description of this map\footnote{Chiswell's paper was preceded by papers of Swan \cite{Swan}
where $(\ref{MayViet})$ is established for amalgamated products, and Bieri \cite{Bie} where $(\ref{MayViet})$ is established for
HNN extensions. But these papers do not contain any explicit description of $\partial$ either.}, although it might of course be known to experts. It is important for our main results to have this map explicitly, where we will also not assume that groups are discrete; we therefore will spend some time developing it. 

\break

Let $\omega \in \hom(\Z(\~A), M)^G$ so that

\begin{eqnarray*}
((\partial \circ s^t)(\omega)(g_1))(g\~v) &=& ((g_1-1)s^t(\omega))(g\~v)\\
%
%
&=& g_1(\int_{\~v_0}^{g_1^{-1}g\~v} \omega ) - \int_{\~v_0}^{g\~v} \omega\\
&=& \int_{g_1\~v_0}^{g\~v} \omega + \int_{g\~v}^{\~v_0}\omega\\
&=& -\int_{\~v_0}^{g\~v_0}\omega,
\end{eqnarray*}
that is, 
\begin{equation}\label{partial}
\~\partial \omega(g) := ((q_*^t)^{-1}\circ \partial \circ ((\delta_*^t)^{-1})(\omega))(g)= -\int_{\~v_0}^{g\~v_0}\omega.
\end{equation}

However, we would like to have a description of $\~\partial$ purely in terms of the graph of groups data.

\begin{Prop}Let $\omega\in\prod_{e\in A} M^{G_e}$.
\begin{enumerate}
\item For $v\in V$ and $g_v\in G_v$, we have: $(\partial \omega)(g_v)=(g_v-1)\int_{v_0}^{v}\omega$.
\item For $e\in A\backslash E(T)$, we have: $(\partial \omega)(t_e)=t_e\int_{v_0}^{e_+}\omega -t_e\omega_e -\int_{v_0}^{e_-}\omega$.
\end{enumerate}\end{Prop}

\begin{proof} Fix $\omega\in\prod_{e\in A} M^{G_e}$; the canonical isomorphism $\prod_{e\in A} M^{G_e}\rightarrow \hom(\Z(\~A), M)^G$ maps $\omega$ to $\~\omega$ where $\~\omega_{g\~e}=g\omega_e$ for $e\in A$ and $g\in G$.
\begin{enumerate}
\item If $g_v\in G_v$ then using the fact that $g_v.\~v = \~v$ we have by (\ref{partial}):

\begin{eqnarray*}
(\partial \omega)(g_v) &=&  -\int_{\~v_0}^{g_v\~v_0}\~\omega \\
&=& \int_{g_v\~v_0}^{\~v}\omega +  \int_{\~v}^{\~v_0}\~\omega\\
&=& (g_v-1)\int_{\~v_0}^{\~v}\~\omega  =  (g_v-1)\int_{v_0}^{v}\omega.
\end{eqnarray*}
where we have used $\pi\~{[v,w]} = [v,w]$, as observed earlier.

\item For $e\in A\backslash E(T)$, we have by (\ref{partial}):
\begin{eqnarray*}
(\~\partial \omega)(t_e) &=&  \int_{t_e\~v_0}^{\~v_0}\~\omega \\
&=&  \int_{t_e\~v_0}^{t_e\~{e_+}}\~\omega + \int_{t_e\~{e_+}}^{\~{e_-}}\~\omega  + \int_{\~{e_-}}^{\~v_0}\~\omega \\
&=&  t_e\int_{\~v_0}^{\~{e_+}}\~\omega + t_e\int_{\~{e_+}}^{t_e^{-1}\~{e_-}}\~\omega  - \int_{\~v_0}^{\~{e_-}}\~\omega .
\end{eqnarray*}
But $\~{e_+}=\~e_+$ and $t_e^{-1}\~{e_-}=\~e_-$ by (\ref{extrorig}), so
$(\~\partial \omega)(t_e)= t_e\int_{v_0}^{{e_+}}\omega  - t_e\int_{e_-}^{{e_+}}\omega -\int_{{e_-}}^{v_0}\omega $. Finally $\int_{e_-}^{e_+}\omega =\omega_e$ by (\ref{edge}), which concludes the proof.
\end{enumerate}
\end{proof}



\begin{Lem}\label{conn}
The class of $\tilde{\partial}\omega$ is independent of choice of the base-vertex $v_0$
\end{Lem}

\begin{proof}
 
 
 
 Consider $\tilde{\partial}'\omega$ defined similarly to $\tilde{\partial}\omega$ but  with respect to a base vertex $v_1$. It is a straightforward computation, which uses Remark \ref{epsilonAcocyle}, that for each $g\in  (\sqcup_{v\in V}G_v)\sqcup \{t_e:e\in A\}$ the difference is a co-boundary. Namely, $$\tilde{\partial}\omega (g) - \tilde{\partial}'\omega (g) = (g-1) \int_{v_0}^{v_1} \omega.$$
\end{proof}


Then, $\partial: \prod_{e\in A}M^{G_e}\rightarrow H^1(G,M)$ as the composition of $\tilde{\partial}$ with the
canonical map $Z^1(G,M)\rightarrow H^1(G,M)$. 



   \section{On the Haagerup Cocycle}\label{HaagerupCoc}

Let $G$ be a topological group acting without inversion on the tree $\mathcal{T}$. Denote by $\mathcal{V}$ and $\E$ the set of vertices and oriented edges, respectively, of the tree  $\mathcal T$. We then have a natural action of $G$ on $\ell^2(\E)$ which we will now study. 
 
Recall that the removal of an (open) edge disconnects the tree into two connected components, called {\it half-trees}. To each $h\in \E$, we will associate the connected component which contains $h_+$. This way, we identify $\E$ with the set of half-trees: as a half-tree, the edge $h$ corresponds to $\{x\in\mathcal{V}: d(x,h_+)<d(x,h_-)\}$. This allows us to write: $x\in h$.

Let $x\in \mathcal V$ and consider the characteristic function $\1_x = \{h\in \E: x \in h\}$ of the set of edges pointing towards $x$. With this notation, we fix an initial vertex $x_0 \in \mathcal V$ and consider $$b(g) = (g-1)\cdot \1_{x_0} = \1_{gx_0} - \1_{x_0}.$$ 
 
Since $b$ is a formal cocycle, the observation that  $\|b(g)\|_2^2 = 2d(gx_0,x_0)$ shows immediately that $[b]\in H^1(G,\ell^2(\E)).$ Furthermore, $b$ is bounded if and only if $G$ has a fixed vertex, and similarly, $G$ acts properly on $\mathcal T$ if and only if $b$ is proper on $G$. We note that the class of $b$ is independent of the base vertex $x_0$. We will call $b$ the \emph{Haagerup cocycle with respect to the base vertex $x_0$}. This cocycle is a witness to the fact that groups that admit a proper action on a tree have the Haagerup property. Note that the class $[b]$  is clearly independent of the choice of $x_0$; this will allow us to choose $x_0$ in an appropriate way, when studying cohomological properties of $b$.

We wish to now understand when this cocycle is trivial in the context of Chiswell's Mayer-Vietoris sequence. Let $X=(V,E)$ be the quotient graph $G\backslash\mathcal{T}$. Fix a base vertex $v_0\in V$, a maximal tree $T$ of $X$, and an orientation $A$ of $E$. 
As before, denote by $\pi:\mathcal{T}\rightarrow X$ the quotient map, with section $x\mapsto \~x$ for $x\in V\cup T$. This lifting  has the property that  $\~{e}_+=\~{e_+}$. Note that $\overline{\~e}={\~{\overline e}}$.

For $e\in A\cap E(T)$, let us denote by $\vartheta_e:G_e\rightarrow G_{e_-}$  the natural inclusion.  For $e\in A\backslash E(T)$, recall that $\~{e}_-= t_e^{-1}\~{e_-}$, and we set $\vartheta_e(g_e)=t_eg_et_e^{-1}$ for $g_e\in G_e$.

We now observe that $b$ is clearly bounded, hence cohomologically trivial on each vertex group $G_v$ ($v\in V$). This means that $[b]$ is in the kernel of $\Delta: H^1(G,\ell^2(\mathcal{E}))\rightarrow \prod_{v\in V} H^1(G_v,\ell^2(\mathcal{E}))$, so by the Mayer-Vietoris sequence (\ref{MayViet}) it is in the image of $\partial$. We explicitly describe an element of $\prod_{e\in A}\ell^2(\E)^{G_e}$ that maps to $[b]$.

Let us define $\omega_e=\delta_{\overline{\~e}}-\delta_{\~e}$ for $e\in A$; observe that $\omega\in\prod_{e\in A}\ell^2(\E)^{G_e}$.

\begin{Lem}\label{omega}
 Let $b \in Z^1(G, \ell^2(\E))$ be the Haagerup cocycle with respect to the base vertex $\~v_0$. Then $\tilde \partial (\omega)= b$, with $\omega$ as above.
\end{Lem}

\begin{proof}
 Observe first that  for $e\in A$:
 $$\omega_e = \1_{{\~e}_-}-\1_{{\~e}_+}. $$
Formally $\omega|_T=df$, where $f_v=-\1_{{\~{v}}}\;(v\in V)$. So:
$$\int_{v_0}^{v} \omega =- \1_{{\~{v}}}+\1_{{\~v_0}}$$
Now, for $g_v\in G_v$ we have: 
\begin{eqnarray*}
\tilde{\partial}\omega(g_v)& = &(g_v-1)\int_{v_0}^{v} \omega\\
&=& (g_v-1)(-\1_{{\~{v}}}+\1_{{\~v_0}}) \\
&=& (g_v-1)\1_{{\~v_0}}\\
&=& b(g_v).
\end{eqnarray*}

Now, let $e\in A$; using ${\~e}_+={\~{e_+}}$ and $t_e({\~e}_-)={\~{e_-}}$: 
\begin{eqnarray*}
(\tilde\partial \omega)(t_e) &=&t_e\int_{v_0}^{e_+} \omega - \int_{v_0}^{e_-} \omega - t_e\omega_e \\
& = & t_e(- \1_{{\~{e_+}}}+ \1_{{\~v_0}}) +  \1_{{\~{e_-}}}- \1_{{\~v_0}} -t_e(\1_{{\~e}_-}-\1_{{\~e}_+})\\
& = &(t_e-1)\1_{{\~v_0}} \\
&= &b(t_e).
\end{eqnarray*}
Since $\tilde \partial \omega$ agrees with $b$ on the generators $\left(\sqcup_{v\in V} G_v  \right)\sqcup\{t_e: e\in A\}$ by Lemma \ref{conn} we have that $\tilde \partial \omega = b$.
\end{proof}

\begin{Lem}\label{watatani} Let $G$ be a topological group acting on a tree $\mathcal T$ without inversion. The following are equivalent:
\begin{enumerate}
\item[a)] $[b]=0$ in $H^1(G,\ell^2(\E))$;
\item[b)] $G$ has a fixed vertex.
\item[c)] $\Delta: H^1(G,\ell^2(\E))\rightarrow \prod_{v\in V} H^1(G_v,\ell^2(\E))$ is injective.
\end{enumerate}
 \end{Lem}
 
 \begin{proof} $(a)\Leftrightarrow(b)$ Using the fact that
 \begin{eqnarray*}
 \|b(g)\|^2 &=& 2d(gx_0,x_0)
\end{eqnarray*}
we see that $[b]=0$ if and only if $b$ is bounded, if and only if $G$ admits a bounded orbit on vertices. By a standard argument (see for example \cite[Proposition 2.7, Chapter II.2]{BridsonHaefliger}), this is equivalent to $G$ fixing a vertex.

$(c)\Rightarrow (a)$ Follows immediately from the already observed fact that $[b]\in \ker\Delta$.

$(b)\Rightarrow (c)$ $G$ admits a globally fixed vertex $x_0$, so that $G=G_{v_0}$, where $v_0=\pi(x_0)$. Since $G$ is equal to a vertex stabilizer, the restriction of $\Delta$ to the corresponding factor of $\prod_v H^1(G_v,\ell^2(\mathcal{E}))$ is the identity, and $\Delta$ is injective.
\end{proof}

\begin{remark} Assume $[b]=0$. Then, by the exact sequence (\ref{MayViet}), the vector $\omega\in \prod_{e\in A}\ell^2(\E)^{G_e}$ from Lemma \ref{omega} is in the image of $\iota:\prod_{v\in V}\ell^2(\mathcal{E})^{G_v}\rightarrow \prod_{e\in A}\ell^2(\E)^{G_e}$. This can be seen explicitly as follows. Suppose $G$ fixes the vertex $x_0$ in $\mathcal{T}$, set $v_0=p(x_0)$. Then $X=G\backslash\mathcal{T}$ is a tree. Set then $f_v=\int_{v_0}^v \omega$; then $f_v\in\ell^2(\E)^{G_v}$, since $G_v$ pointwise fixes the geodesic $[x_0,{\~{v}}]$; and clearly $\iota(f)=df=\omega$ (which is essentially Poincar\'e's lemma).
\end{remark}

From Lemma \ref{watatani} and the Delorme-Guichardet Theorem, we immediately deduce the following consequence, that provides a cohomological characterization of Serre's Property (FA): 

\begin{Cor} A topological group $G$ has Serre's Property (FA) if and only, for any action of $G$ without inversion on a tree $\mathcal{T}$, the map $\Delta: H^1(G,\ell^2(\E))\rightarrow\prod_{v\in V} H^1(G_v,\ell^2(\E))$ is injective. In particular, if $G$ is locally compact with Kazhdan's Property (T), then $G$ has Serre's Property (FA)\footnote{The latter statement was first proved by Watatani \cite{Wat}.}.
\hfill$\square$
\end{Cor}

We now explore the triviality of the Haagerup cocycle in reduced cohomology.

\begin{Lem}\label{FreeF_2}
 Suppose that $\mathbb{F}_2 \leq \mathrm{Aut}(\mathcal T)$ acts freely. Then $\ell^2(\E)$ does not have $\mathbb{F}_2$-almost invariant vectors.\end{Lem}

\begin{proof} Let $C$ denote a choice of one representative in each $\mathbb{F}_2$-orbit in $\mathcal{E}$. Since the $\mathbb{F}_2$-action is free, this choice identifies the $\ell^2(\mathcal{E}) \cong  \underset{c\in C}{\oplus}\ell^2(\mathbb{F}_2)$, as an $\mathbb{F}_2$-module, where the direct sum is endowed with the diagonal left regular representation of $\mathbb{F}_2$. So the result follows from the observation that $\ell^2(\mathbb{F}_2)$ does not have $\mathbb{F}_2$-almost invariant vectors, which is guaranteed by the non-amenability of $\mathbb{F}_2$.
\end{proof}

\begin{Def}
 A group $G$ acting on a tree $\mathcal T$ is said to be elementary if it has a finite orbit in $\mathcal T$ or $\partial \mathcal T$. 
\end{Def}

If one has a finite orbit in $\partial \mathcal T$ but not in $\mathcal T$, then the orbit must have size at most 2. This follows for example from Propositions 1 and 2 of \cite{PaysValette} along with the classification of isometries. 

We now recall a little about the structure of the stabilizers $Aut(\mathcal T)_{\xi_0}$ and $Aut(\mathcal T)_{\{\xi_0, \xi_1\}}$ where $\xi_0, \xi_1\in \partial \mathcal T$. The group $Aut(\mathcal T)_{\xi_0}$ contains the collection of its elliptic elements as a normal subgroup  $R_{\xi_0} = \cup_{v\in \mathcal T} stab[v, \xi_0)$. 
The map that associates to each element in $Aut(\mathcal T)_{\xi_0}$ its signed translation length is a homomorphism to $\mathbb{Z}$ with kernel $R_{\xi_0}$, see \cite[Lemme 4]{PaysValette}; we call it the {\it Busemann homomorphism}.  Choosing $a\in Aut(\mathcal T)_{\xi_0}$ a hyperbolic element with minimal translation length (or setting $a=1$ otherwise) describes  an isomorphism $$Aut(\mathcal T)_{\xi_0} \cong \<a\>\ltimes R_{\xi_0}.$$
This provides $Aut(\mathcal T)_{\xi_0}$ with a normal form, i.e. for each $g\in Aut(\mathcal T)_{\xi_0}$ there is a unique $n\in \Z$ and $r\in R_{\xi_0}$ such that $g=a^n r$.

Next, consider $G=Aut(\mathcal T)_{\{\xi_0, \xi_1\}}$; observe that it contains, as a subgroup of index at most two, $G_0=Aut(\mathcal T)_{\xi_0}\cap Aut(\mathcal T)_{\xi_1}$ and $G_0\cong \<a\> \ltimes (R_{\xi_0}\cap R_{\xi_1})$. 

Finally, we observe that these descriptions and canonical forms hold by restriction to any subgroup or $Aut(\mathcal T)_{\xi_0}$ or $Aut(\mathcal T)_{\{\xi_0, \xi_1\}}$.

To simplify notation which will quickly become cumbersome, let $$\mathbbm{2}_{x,y} = \1_{[x,y]}-\1_{[y,x]}.$$ The fact that $\|\mathbbm{2}_{x,y}\|^2 = 2d(x,y)$ should give the reader an idea of why the notation was chosen this way. Observe that if $b$ is the Haagerup cocycle with respect to base point $x_0$ then $$b(g) = \mathbbm{2}_{x_0,gx_0}.$$

\begin{Thm}\label{Haagerupelementary}
Let $G$ be a discrete group acting on $\mathcal T$ without inversion. Let $b\in Z^1(G, \ell^2(\E))$ denote the Haagerup cocycle. Then $[b]$ is trivial in $\overline H^1(G, \ell^2(\E))$ if and only if the $G$-action is elementary.
\end{Thm}

\begin{proof}
 
Assume the $G$-action is non-elementary. Then, by \cite{PaysValette} there exists a freely acting $F_2\leq G$. By Lemma \ref{FreeF_2}, $\ell^2(\mathcal{E})$ has no $F_2$-almost invariant vectors, hence no $G$-almost invariant vectors. By Guichardet's result: $\overline{H}^1(G,\ell^2(\mathcal{E}))=H^1(G,\ell^2(\mathcal{E}))$; so it is enough to show that $[b]\neq 0$ in $H^1(G,\ell^2(\mathcal{E}))$, i.e. that $b$ is unbounded on $G$. But $b$ is already unbounded on $F_2$, as it acts freely.

Conversely, suppose that the action is elementary. If there is a finite orbit  in $\mathcal{T}$, then by Lemma \ref{watatani}, $[b]$ is trivial in $H^1(G, \ell^2(\E))$ and hence in $\overline H^1(G, \ell^2(\E))$.

Therefore, assume that $G$ does not have a finite orbit in $\mathcal T$. Then, either $G$ has a fixed point in $ \partial \mathcal T$ or there is a $G$-invariant set $\{\xi_0,\xi_1\}\subset \partial \mathcal T$ such that $G_0: = G\cap Aut(\mathcal T)_{\xi_0}\cap Aut(\mathcal T)_{\xi_1}$  has index 2 in $G$.

\noindent
\textbf{Case 1: \ } $G$ fixes $\xi_0$. 

 Chosing a hyperbolic isometry $a\in G$ of minimal translation length $\ell(a)$, every element of $G$ may be described uniquely as $a^N r$ for $N\in \Z$ and $r\in R_{\xi_0}\cap G$. Replacing $a$ by $a^{-1}$ if necessary, we may assume that $\xi_0$ is a contracting fixed point for $a$.

 Let $F\subset G$ be a finite set. Then $F\subset \{a^{N}r: M'\leq N\leq M, r\in F_0\}$ where $F_0$ is a finite subset of $R_{\xi_0}$, with $1\in F_0$. We begin by considering the case where   $F= \{a^Nr: 0\leq N\leq M, r\in F_0\}$.

 Let $A$ be the axis of $a$. The elements of the finite set $F_0\subset R_{\xi_0}$
  
 must have a common fixed point $t$ which allows us to choose $x_0 \in [t,\xi_0)\cap A$, that we take as base-point for the Haagerup cocycle.

To simplify notation, let $x_n = a^nx_0$ for $n\in \Z$ and observe that if $\ell(a)$ is the translation length of $a$ then $\|\mathbbm{2}_{x_k,x_{k+1}}
\|^2 = 2\ell(a)$. With this, we have that, for $N\in \Z$ and $r\in F_0$ $$b(a^Nr) =\mathbbm{2}_{x_0,x_{N}} 
.$$

Let $v_n= -\overset{n}{\Sum{k = 0}} (1-\frac{k}{n})\mathbbm{2}_{x_k,x_{k+1}}
$ and $b_n(g) = gv_n - v_n$. If $0\leq N\leq M$, $n>M$, $r\in F_0$ then

\begin{eqnarray*}
b_n(a^Nr) 
&=& -\overset{n+N}{\Sum{k = N}} (1-\frac{k-N}{n})\mathbbm{2}_{x_k,x_{k+1}}
+ \overset{n}{\Sum{k = 0}} (1-\frac{k}{n})\mathbbm{2}_{x_k,x_{k+1}} 
\\
&=& \overset{N-1}{\Sum{k = 0}} (1-\frac{k}{n})\mathbbm{2}_{x_k,x_{k+1}} 
-\overset{n}{\Sum{k = N}} \frac{N}{n}\mathbbm{2}_{x_k,x_{k+1}} 
\\
&& -\overset{n+N}{\Sum{k = n+1}} (1-\frac{k-N}{n})\mathbbm{2}_{x_k,x_{k+1}}
\end{eqnarray*}

So, observing that $b(a^Nr) = \overset{N-1}{\Sum{k=0}}\mathbbm{2}_{x_k,x_{k+1}}
$ and that $\mathbbm{2}_{x_k,x_k+1}$ is orthogonal to $\mathbbm{2}_{x_{k'},x_{k'+1}}$ for $k\neq k'$:

\begin{eqnarray*}
\|b(a^{N}r) - b_n(a^Nr)\|^2  
&= & \|  \overset{N-1}{\Sum{k=0}}\mathbbm{2}_{x_k,x_{k+1}} -
\overset{N-1}{\Sum{k = 0}} (1-\frac{k}{n})\mathbbm{2}_{x_k,x_{k+1}} 
\\
&&\qquad
+\overset{n}{\Sum{k = N}} \frac{N}{n}\mathbbm{2}_{x_k,x_{k+1}} 
+\overset{n+N}{\Sum{k = n+1}} (1-\frac{k-N}{n})\mathbbm{2}_{x_k,x_{k+1}}\|^2 \\
&=& \| \overset{N-1}{\Sum{k=0}}\frac{k}{n}\mathbbm{2}_{x_k,x_{k+1}} +\frac{N}{n} \overset{n}{\Sum{k=N}}\mathbbm{2}_{x_k,x_{k+1}} +   \overset{n+N}{\Sum{k=n+1}}(1-\frac{k-N}{n})\mathbbm{2}_{x_k,x_{k+1}} \|^2\\
&=&2\ell(a)\left[ \frac{1}{n^2}\overset{N-1}{\Sum{k=1}}k^2 + \frac{N^2}{n^2}(n-N+1) + \frac{1}{n^2}\overset{n+N}{\Sum{k=n+1}}(n+N-k)^2\right]\\
&\leq& \frac{2\ell(a)}{n^2}\left(M^3 + M^2n +M^3\right)\\
&\leq&\frac{6\ell(a)M^2}{n}
 \end{eqnarray*}

This part is concluded by observing that, if $K$ is an arbitrary finite subset of $G$, then for $N\gg 0$, the set $a^N K$ is contained in a finite set $F$ of the above form. Defining the 1-cocycle $c_n$ as $c_n=b-b_n$, we then have, for $h\in F$:
$$c_n(a^{-N}h)=a^{-N}(c_n(h)-c_n(a^{N})),$$
so by the triangle inequality: $  \|c_n(g)\|\leq \sqrt{\frac{12\ell(a)}{n}}M\underset{n\to \infty}{\longrightarrow}0$ for every $g\in K$.

\noindent
\textbf{Case 2: \ } $G$ does not have a fixed point in $\partial\mathcal T$ but preserves $\{\xi_0, \xi_1\}\subset \partial \mathcal T$.

Let $G_0 = G\cap Aut(\mathcal T)_{\xi_0}\cap Aut(\mathcal T)_{\xi_1} $ and observe that $G_0$ has index 2 in $G$. By the first case, $[b]$ is trivial in $\overline{H}^1(G_0, \ell^2(\E))$. By lemma \ref{injective} just below, $[b]$ is also trivial in $\overline{H}^1(G, \ell^2(\E))$

\end{proof}

\begin{Lem}\label{injective} Let $H$ be a finite index subgroup in the discrete group $G$. For any unitary $G$-module $M$, the restriction map $Rest_G^H:\overline{H}^1(G,M)\rightarrow \overline{H}^1(H,M)$ is injective.
\end{Lem}

\begin{proof} Let $g_1,...,g_N$ be representatives for the left cosets of $H$ in $G$. Let $b\in Z^1(G,M)$ be a 1-cocycle such that $b|_H$ is a limit of coboundaries. We must show that $b$ is a limit of coboundaries. Passing to the associated affine action $\alpha(g)v=gv + b(g)\;(g\in G,v\in M)$: under the assumption that there is a sequence $(v_k)_{k>0}\in M$ such that $\lim_{k\rightarrow\infty}\|\alpha(h)v_k - v_k\|=0$ for every $h\in H$, we must shot the existence of a sequence $(w_k)_{k>0}\in M$ such that $\lim_{k\rightarrow\infty}\|\alpha(g)w_k - w_k\|=0$ for every $g\in G$. So, fix $g\in G$. There exists a permutation $\sigma$ of $\{1,2,...,N\}$ and elements $h_1,...,h_N \in H$ such that $gg_i=g_{\sigma(i)}h_i$ for every $i=1,...,N$. Set $w_k=\frac{1}{N}\sum_{i=1}^N \alpha(g_i)v_k$. Then, using that $\alpha(s)x-\alpha(s)y=s(x-y)$ for every $s\in G,\;x,y\in M$:
$$\alpha(g)w_k-w_k=\frac{1}{N}(\sum_{i=1}^N \alpha(gg_i)v_k)-w_k=\frac{1}{N}\sum_{i=1}^N\alpha(g_{\sigma(i)}h_i)v_k - \frac{1}{N}\sum_{i=1}^N \alpha(g_{\sigma(i)})v_k$$
$$=\frac{1}{N}\sum_{i=1}^N g_{\sigma(i)}(\alpha(h_i)v_k-v_k).$$
Since $\lim_{k\rightarrow\infty}\|\alpha(h_i)v_k - v_k\|=0$ for $i=1,...,N$, we deduce that $\lim_{k\rightarrow\infty}\|\alpha(g)w_k - w_k\|=0$.
\end{proof}
 
 \section{Proof of Theorem \ref{main} and one application}\label{PfMainTh}
 
 \subsection{Proof of Theorem \ref{main}(ii)}
 
We first assume that $\iota:\prod_{v\in V}M^{G_v}\rightarrow\prod_{e\in A}
 M^{G_e}$ has dense image. Now, observing the explicit formula for $\tilde{\partial}$ we see that $\partial:\prod_{e\in A}M^{G_e}\rightarrow H^1(G,M)$ is continuous for the product topology. Furthermore, the assumption of Property (T) for the vertex groups implies that $H^1(G_v,M)= 0$  for each $v$. This means that $\partial$ is onto, by the Mayer-Vietoris sequence $(\ref{MayViet})$. We therefore have that $\tilde\partial:\prod_{e\in A}M^{G_e}\rightarrow Z^1(G,M)$ is onto. Furthermore, $\im(\iota)= \ker(\partial)$ is dense in $\prod_{e\in A}M^{G_e}$, which means $\tilde\partial|_{\ker(\partial)} $ has dense image in $B^1(G,M)$. This of course means that $B^1(G,M)$ is dense in $Z^1(G,M)$ and hence $\overline{H}^1(G,M)=0$.

 Conversely, assume that $\overline{H}^1(G,M)=0$. Continuing to assume that all vertex groups have Property (T), the Mayer-Vietoris sequence yields that $\tilde\partial: \prod_{e\in A}M^{G_e} \to Z^1(G,M)$ is onto and that $\im(\iota) =\ker(\partial)$. Therefore, choosing $\omega\in \prod_{e\in A}M^{G_e}$, we must show that $\omega$ can be approximated by elements in the image of $\iota$. By assumption, there exists a sequence $(m_k)_{k\geq 1}$ of vectors in $M$ such that $\tilde{\partial}\omega(g)=\lim_{k\rightarrow\infty}gm_k-m_k$. By definition of $\tilde{\partial}\omega|_{G_v}$, the sequence $(-m_k + \int_{v_0}^v \omega)_{k\geq 1}$ is almost $G_v$-invariant, for all $v\in V$. Denote by $P_v$ the orthogonal projection of $M$ onto $M^{G_v}$, and define $f_k\in\prod_{v\in V} M^{G_v}$ by $(f_k)_v=P_v(-m_k + \int_{v_0}^v \omega)$ (for $v\in V$). 
 
  {\bf Claim:} $\lim_{k\rightarrow\infty}\|(f_k)_v-(-m_k + \int_{v_0}^v \omega)\|=0$ for every $v\in V$.
 
 Suppose not, for some $v\in V$. Passing to a subsequence, we may assume that $\|(f_k)_v-(-m_k + \int_{v_0}^v \omega)\|$ is bounded below by a positive constant. But $(f_k)_v-(-m_k + \int_{v_0}^v \omega)$ belongs to the orthogonal complement $(M^{G_v})^\perp$. So the sequence $\left(\frac{(f_k)_v-(-m_k + \int_{v_0}^v \omega)}{\|(f_k)_v-(-m_k + \int_{v_0}^v \omega)\|}\right)_{k\geq 1}$ is an almost $G_v$-invariant sequence of unit vectors in  $(M^{G_v})^\perp$, which clearly has no non-zero $G_v$-invariant vector. This contradicts Property (T) for $G_v$, establishing the claim.
 
 \medskip
 The proof of the theorem is then finished by showing that $\omega=\lim_{k\rightarrow\infty} \iota(f_k)$. But, for $e\in A$:
 \begin{eqnarray*}
\omega_e-\iota(f_k)_e&=&\omega_e-(f_k)_{e_+}+t_e^{-1}(f_k)_{e_-} \\
 &=&\omega_e-(-m_k+\int_{v_0}^{e_+}\omega) + t_e^{-1}(-m_k+\int_{v_0}^{e_-}\omega)\\
&&+[(-m_k+\int_{v_0}^{e_+}\omega) -(f_k)_{e_+}]+t_e^{-1}[(f_k)_{e_-}-(-m_k+\int_{v_0}^{e_-}\omega)].
\end{eqnarray*}

 By the Claim, the two terms in brackets go to 0 for $k\rightarrow\infty$. It remains to show that $\lim_{k\rightarrow\infty}\|\omega_e-(-m_k+\int_{v_0}^{e_+}\omega) + t_e^{-1}(-m_k+\int_{v_0}^{e_-}\omega)\|=0$. But 

\begin{eqnarray*}
\omega_e-(-m_k+\int_{v_0}^{e_+}\omega) &+& t_e^{-1}(-m_k+\int_{v_0}^{e_-}\omega)\\
&=& \omega_e-t_e^{-1}(m_k-t_em_k)+t_e^{-1}(\int_{v_0}^{e_-}\omega)-\int_{v_0}^{e_+}\omega\\
&=& \omega_e -t_e^{-1}(m_k-t_em_k) - t_e^{-1}(\tilde\partial \omega(t_e)) -\omega_e\\
&=& -t_e^{-1}[\tilde\partial\omega(t_e) + (m_k -t_em_k)]\to 0
\end{eqnarray*}
where the last line converges to 0 by assumption. This concludes the proof.
 \hfill$\square$
 


 




  \subsection{The case of HNN-extensions}\label{HNN}

Let $G=HNN(\Gamma,A,\theta)$ be an HNN-extension, where $A$ is a subgroup of $\Gamma$ and $\theta:A\rightarrow\Gamma$ is a monomorphism.
Recall from \cite{Serre} that $G$ can be seen as the fundamental group of a graph of groups with one vertex, with group $\Gamma$, and one 
edge, with group $A$. Let $t$ be the stable letter in $G$ corresponding to the unique edge, satisfying $tat^{-1}=\theta(a)$, for
every $a\in A$. If $A=\Gamma$ and $\theta$ is an automorphism of $\Gamma$, then $G$ is the semi-direct product $\Gamma\rtimes_{\theta}\mathbb{Z}$.

The map $\iota:M^{\Gamma}\rightarrow M^A$ is given by $m\mapsto (1-t^{-1})m$.

\medskip

\begin{proof}[Proof of Corollary \ref{semidir}] First observe that $Ker (1-t)\cap M^{\Gamma} = 
Ker (1-t^{-1})\cap M^{\Gamma}=M^G$, as $\Gamma \cup \{t\}$ generates $G$. In other words, 1 is not an eigenvalue of $t|_{M^{\Gamma}}$
if and only if $M^G=0$.

{\bf Claim:} Let $M$ be a unitary $G$-module; if $M^G\neq 0$, then $\overline{H}^1(G,M)\neq 0$ (in particular, $H^1(G,M)\neq 0$). 

To see it, let $M^\perp$ be the orthogonal of $M^G$ in $M$; from the decomposition $M=M^G\oplus M^\perp$ we get a decomposition $\overline{H}^1(G,M)=\overline{H}^1(G,M^G)\oplus\overline{H}^1(G,M^\perp)$, and it is enough to check that $\overline{H}^1(G,M^G)\neq 0$. But $\overline{H}^1(G,M^G)=H^1(G,M^G)=\hom(G,M^G)$, which is non-zero as $G$ maps onto $\mathbb{Z}$.

We may now prove the first statement of Corollary \ref{semidir}. Assume first that $H^1(G,M)=0$. By Theorem 1, the map 
$(1-t^{-1})|_{M^{\Gamma}}$ is then onto.
By the previous claim, it follows that $M^G= 0$ and so $\iota = (1-t^{-1})|_{M^{\Gamma}}$ is also injective and therefore invertible, meaning that 1 is not a spectral value of $t|_{M^\Gamma}$. 
Conversely, if 1 is not a spectral value of $t|_{M^\Gamma}$, then $(1-t^{-1})|_{M^{\Gamma}}$ is invertible, in particular it is onto, so 
$H^1(G,M)=0$.
\medskip

We now pass to the second statement of Corollary \ref{semidir}. If $\overline{H}^1(G,M)=0$, then by the claim, 1 is not an eigenvalue of $t|_{M^\Gamma}$.
Conversely, if 1 is not an eigenvalue, then $(Im((1-t^{-1})|_{M^{\Gamma}}))^{\perp}=Ker((1-t^{-1})|_{M^{\Gamma}})=0$, i.e. $Im((1-t^{-1})|_{M^{\Gamma}})$
is dense, so $\overline{H}^1(G,M)=0$ by Theorem \ref{main}.
\end{proof}


 \section{The first $\ell^2$-Betti number}\label{SectL2Betti}
 

\subsection{Computing $\ell^2$-Betti numbers}


Let $G$ be a countably infinite group acting without inversion and co-compactly on a tree ${\cal T}$, with quotient graph $X=(V,E)=G\backslash{\cal T}$.  

Let $EG$ be a contractible CW-complex endowed with a proper, free $G$-action. For a $G$-CW-complex $Z$, we may define $$\overline{H}^i_{(2)}(Z;G):=\overline{H}^i_{(2)}(Z\times EG;G)$$ (see \cite[Proposition 2.2]{Cheeger-Gromov}) using the fact that the action of $G$ on $Z\times EG$ is now free. We denote by $\beta^i(G):=\dim_G\overline{H}^i_{(2)}(EG,G)$ the $i$-th $L^2$-Betti number. 

Let $Y$ be the geometric realization of ${\cal T}$, so that $Y$ is a contractible, 1-dimensional CW-complex. Let $Y'$ be the set of vertices of ${\cal T}$, viewed as a sub-complex of $Y$. Recall that the relative $L^2$-cohomology $\overline{H}^i_{(2)}(Y,Y';G)$ is the cohomology of the complex $$C^*_{(2)}(Y,Y';G):=\ker[C^*_{(2)}(Y\times EG)\stackrel{\mbox{Rest}}{\longrightarrow} C^*_{(2)}(Y'\times EG)].$$ 

\begin{Lem}\label{computL2}
\begin{itemize}
\item[i)] $\dim_G \overline{H}^i_{(2)}(Y;G)=\beta^i(G)$ for $i\geq 0$.
\item[ii)] $\dim_G \overline{H}^i_{(2)}(Y',G)=\sum_{v\in V}\beta^i(G_v)$ for $i\geq 0$.
\item[iii)] $\overline{H}^0_{(2)}(Y,Y';G)=0$ and $\dim_G \overline{H}^i_{(2)}(Y,Y';G)=\sum_{e\in A} \beta^{i-1}(G_e)$ for $i\geq 1$.
\end{itemize}
\end{Lem}

\begin{proof}
\begin{itemize}
\item[i)] Since $Y$ is contractible, $Y\times EG$ is a contractible CW-complex on which $G$ acts properly freely. By uniqueness of $EG$, the space $Y\times EG$ is $G$-equivariantly homotopic to $EG$. So $$\dim_G \overline{H}^i_{(2)}(Y;G)=\dim_G\overline{H}^i_{(2)}(Y\times EG;G)=\beta^i(G).$$
\item[ii)] Choosing a vertex $\tilde{v}$ in each $G$-orbit of $Y'$, we get 
\begin{eqnarray*}
\dim_G \overline{H}^i_{(2)}(Y',G)&=&\sum_{v\in V} \dim_G \overline{H}^i_{(2)}(G\cdot\tilde{v};G)\\
&=&\sum_{v\in V}\dim_{G_{\tilde{v}}} \overline{H}^i_{(2)}(\tilde{v};G_{\tilde{v}})=\sum_{v\in V} \beta^i(G_v)
\end{eqnarray*}
 where the previous to last equality is \cite[Proposition 2.5]{Cheeger-Gromov}.
 \item[iii)] In degree 0, we have $C^0_{(2)}(Y,Y';G)=0$, as $Y\times EG$ and $Y'\times EG$ have the same vertices. In degree $i\geq 1$, denote by $Z^{(i)}$ the set of $i$-cells of the CW-complex $Z$. Observe that $$(Y\times EG)^{(i)}=\coprod_{k=0}^i (Y^{(k)}\times EG^{(i-k)})=(Y^{(0)}\times EG^{(i)})\amalg (Y^{(1)}\times EG^{(i-1)})$$
 as $Y$ is 1-dimensional. So
 $$C^i_{(2)}(Y,Y';G)=\ell^2(Y^{(1)}\times EG^{(i-1)}).$$ 
and the co-boundary operator $d^{(i)}: C^i_{(2)}(Y,Y';G)\rightarrow C^{i+1}_{(2)}(Y,Y';G)$ coincides with $1\otimes d^{(i-1)}$. Therefore $\dim_G \overline{H}^i(Y,Y';G)=\dim_G \ker(1\otimes d^{(i-1)})= \sum_{e\in A}\beta^{i-1}(G_e)$ by an argument similar to Part $(ii)$ above (by choosing one representative for each $G$-orbit in $Y^{(1)}$).
\end{itemize}
\end{proof}

The second part of Proposition \ref{L2Betti EQ} below, on amenable vertex-groups, was first obtained by J. Schafer (\cite{Sch}, Corollary 3.12, (ii)). 

\begin{Prop}\label{L2Betti EQ} Assume that, for every vertex $v$ of ${\cal T}$ the stabilizer $G_v$ satisfies $\beta^i(G_v)=0$ for $i\geq 1$. Then
$$\beta^1(G)=\sum_{e\in A}\frac{1}{|G_e|}-\sum_{v\in V}\frac{1}{|G_v|}$$ and $\beta^i(G)=\sum_{e\in A}\beta^{i-1}(G_e)$ for $i\geq 2$.
In particular, if $G_v$ is amenable for every $v\in V$, then $\beta^i(G)=0$ for $i\geq 2$.
\end{Prop}

\noindent
\begin{proof} According to \cite[Lemma 2.3]{Cheeger-Gromov}, the relative $L^2$-cohomology sequence:
$$0\rightarrow \overline{H}^0_{(2)}(Y,Y';G)\rightarrow \overline{H}^0_{(2)}(Y;G)\rightarrow\overline{H}^0_{(2)}(Y';G)$$
$$\rightarrow \overline{H}^1_{(2)}(Y,Y';G)\rightarrow \overline{H}^1_{(2)}(Y;G)\rightarrow\overline{H}^1_{(2)}(Y';G)\rightarrow...$$
is weakly exact. Then by the rank theorem for von Neumann $G$-dimension, whenever some space has $G$-dimension 0, the alternate sum of the $G$-dimensions of the previous terms vanishes. The first statement then follows immediately from Lemma \ref{computL2}. If all vertex-groups are amenable, then so are all edge-groups, hence $\beta^{i-1}(G_e)=0$ for $i\geq 2$ and $e\in A$.
\end{proof}

\begin{Ex}\begin{itemize}
\item[a)] The Baumslag-Solitar group $BS(1,2)$ is the solvable group with presentation
$$BS(1,2)=<a,b|aba^{-1}=b^2>.$$
Consider then the group $H$ with presentation 
$$H=<a_0,a_1,a_2|a_0a_1a_0^{-1}=a_1^2; \;a_1a_2a_1^{-1}=a_2^2>.$$
Clearly $H$ is the amalgamated product of two copies of $BS(1,2)$ over $\Z$:
$$H=BS(1,2)\ast_\Z BS(1,2),$$
so $\beta^i(H)=0$ for every $i\geq 0$ by Proposition \ref{L2Betti EQ}.
\item[b)] Consider the famous Higman group $H_4$ (see \cite{Hig}) with its presentation on 4 generators and 4 relations:
$$H_4=<a_0, a_1, a_2, a_3|a_ia_{i+1}a_i^{-1}=a_{i+1}^2, i\in\Z/4\Z>.$$
Then the subgroups $<a_0,a_1,a_2>$ and $<a_2,a_3,a_0>$ are both isomorphic to the group $H$ above, while $<a_0,a_2>$ is free of rank 2, and $H_4$ is an amalgamated product of two copies of $H$ over the free group $\mathbb{F}_2$:
$$H_4=H\ast_{\mathbb{F}_2}H.$$
By Proposition \ref{L2Betti EQ} we get $\beta^i(H_4)=\beta^{i-1}(\mathbb{F}_2)$ for $i\geq 2$, i.e.
$$\beta^i(H_4)=\left\{\begin{array}{ccc}0 & if & i\neq 2 \\1 & if & i=2\end{array}\right.$$
\end{itemize}
\end{Ex}

\subsection{Sufficient conditions for vanishing and non-vanishing}

The following result will be important for the treatment of reduced graphs of groups in the next section. Note that the assumption is satisfied if vertex groups have property (T).

\begin{Prop}\label{VanishH^1}
Let $G$ be a graph of groups with at least one edge, such that all vertex groups satisfy $H^1(G_v,\ell^2(G))=0$. The following are true:
\begin{enumerate}
 \item If for every edge $e$ we have $|G_e|=\8$, then $H^1(G, \ell^2(G)) = 0$ and $\beta^1(G)=0$.
 \item  If there is an edge $e\in A$ such that  $|G_e|<\8$ and $ e_+ =e_-$ then $H^1(G, \ell^2(G)) \neq 0$ If moreover $G$ is non-amenable, then $\beta^1(G)>0$. 
  \item If there is an edge $e\in A$ such that $|G_e|<\8$, $e_+ \neq e_-$ and both $[G_{e_+}:G_e]\geq 2 $ and $  [G_{e_-}:t_eG_et_e^{-1}]\geq 2$  then $H^1(G, \ell^2(G)) \neq 0$ If moreover $G$ is non-amenable, then $\beta^1(G)>0$. 
\end{enumerate}
\end{Prop}

\vskip.15in
\noindent
\begin{proof}(1) By Chiswell's sequence (\ref{MayViet}), we have that $H^1(G,\ell^2(G))=0$ if and only if the map $\iota : \prod_{v\in V}\ell^2(G)^{G_v} \to \prod_{e\in A}\ell^2(G)^{G_e}$ is onto. 

If for every $e\in A$  we have that $|G_e|=\8$   then $\ell^2(G)^{G_e} = \{0\}$ and $\iota$ is onto. 

(2) Assume that for every $e\in A$ we have $|G_e|<\8$ and $e_+=e_-$. By contradiction assume that $H^1(G,\ell^2(G))= \{0\}$ then  $\iota$ is onto by Chiswell's sequence (\ref{MayViet}), and in particular, there is an $f$ such that $\iota(f)_e=\chi_{G_e}$. Let $v:=e_+=e_-$. Then, we have that $ f_v(x)-f_v(t_e x) = \chi_{G_e}(x)$, in particular if $x\notin G_e$ then $f_v(x) - f_v(t_e x)= 0$ i.e. $f_v(x) = f_v(t_e x)$. Taking $x= t_e^{\pm n}$ for $n\in \mathbb{N}$, a straightforward induction shows that if $n\geq 1$ then $f_v(t_e^n) = f_v(t_e)$ and $f_v(t_e^{-n}) =f_v(1)$.  Since $f_v\in \ell^2(G)$ and $\<t_e\>$ is an infinite subgroup of $G$, we conclude that $f_v(1)= f_v(t_e)=0$.

On the other hand, $f_v(1)-f_v(t_e) = \chi_{G_e}(1) =1$ which means that either $f_v(1)\neq0$ or $f_v(t_e)\neq0$, a contradiction.

\medskip
\noindent
(3) Assume that there is  $e\in A$ such that $|G_e|<\8$, $e_+\neq e_-$ and 
 $[G_{e_+}:G_e]\geq 2$ and $[G_{e_-}:t_eG_et_e^{-1}]\geq 2$. 
\vskip.15in
\noindent
Fix such an $e$ and set $v:=e_+$ and $u=e_-$. We may then take $e$ to be in the maximal spanning tree of the quotient graph so that $G_e \leq G_v \cap G_u$. 

 Observe that $\iota^{-1}(\ell^2(G)^{G_e}) = \ell^2(G)^{G_v}\oplus \,\ell^2(G)^{G_u}$. By contradiction, assume that $\iota (f_v,f_u) = f_v-f_u = \chi_{G_e} \text{ for some }(f_v,f_u) \in \ell^2(G)^{G_v}\oplus \ell^2(G)^{G_u}.$
This means that  $f_v(x) = f_u(x)$ for every $x\notin G_e$. 

By assumption, there is a $g_v\in G_v\setminus G_e$ and a $g_u\in G_u\setminus G_e$. Then, $g_vg_u$ is a hyperbolic isometry of the tree (from which the graph of groups decomposition comes). This means that for each $n\in \mathbb N$ the element  $(g_vg_u)^n \in G$ is distinct and not in $ G_e$. We claim that $f_v((g_vg_u)^{-n}) = f_v(1)$ for every $n\in \mathbb N$. Assume $n=1$. Then, since $(g_vg_u)^n \notin G_e$, and $f_u$ and $f_v$ are $G_u$ and $G_v$-invariant respectively, we have that

\begin{eqnarray*}
f_v(g_u^{-1}g_v^{-1}) &=&  f_u(g_u^{-1}g_v^{-1})\\
&=&  f_u(g_v^{-1})\\
&=&  f_v(g_v^{-1})\\
&=&  f_v(1)\\
\end{eqnarray*}

\noindent
Assume that $f_v((g_vg_u)^{-n}) = f_v(1)$. Then, again, we have that $(g_vg_u)^{n+1} \notin G_e$ and so

\begin{eqnarray*}
f_v(g_u^{-1}g_v^{-1}(g_vg_u)^{-n}) &=&  f_u(g_u^{-1}g_v^{-1}(g_vg_u)^{-n})\\
&=&  f_u(g_v^{-1}(g_vg_u)^{-n})\\
&=&  f_v(g_v^{-1}(g_vg_u)^{-n})\\
&=&  f_v((g_vg_u)^{-n})\\
&=&  f_v(1)\\
\end{eqnarray*}

Therefore, the set $\{g\in G: f_v(g) = f_v(1)\}$ is infinite. This means that $f_v(1)=0$. A similar argument shows that $f_u(1)=0$. But this is impossible as $f_v(1)-f_u(1)=\chi_{G_e}(1) = \chi_{G_e}(1) =1$, a contradiction. Therefore, $\iota$ is not onto. 

\medskip
The statements regarding $\beta^1(G)$ follow from one of the possible definitions for $\beta^1(G)$, namely the von Neumann dimension of $\overline{H}^1(G,\ell^2(G))$ (see Definition 1.30 in \cite{Leuck}), together with Guichardet's classical result that  $\overline{H}^1(G,\ell^2(G))=H^1(G,\ell^2(G))$ when $G$ is non-amenable.



\end{proof}

When vertex stabilizers $G_v$ are non-amenable with $\beta^1(G_v)=0$, part (1) of Proposition \ref{VanishH^1} appears as Theorem 4.1 in \cite{MV07}.

\subsection{Reduced graphs of groups}\label{redgraphs}

\begin{Def}[\cite{deCornulier}] \label{reduced} A graph of groups is said to be \emph{reduced} if  whenever $e\in E$ such that $e_+ \neq e_-$ we have that $[G_{e_+}: G_e]\geq 2$ and $[G_{e_-}:t_eG_et_e^{-1}]\geq 2.$ Otherwise, it is said to be 
\emph{unreduced}.
 \end{Def}

As is pointed out in \cite{deCornulier}, one may pass from an unreduced graph of groups to a  reduced one simply by retracting edges $e\in E$ such that $e_+\neq e_-$ and $G_e =  G_{e_+}$ without affecting the isomorphism type of the group.

We make the important observation that the  cases of Proposition \ref{VanishH^1} account for all possibilities whenever the graph of groups is reduced. 
Indeed either $|G_e| = \8$ for every $e\in A$ (which is case (1)), or there is an edge $e\in A$ such that $|G_e|<\8$. Then either $e_+=e_-$ and we are in case (2) or for every $e\in A$ such that $|G_e|<\8$ we must have that $e_+\neq e_-$. For such edge $e$ we must have $[G_{e_+}:G_e]\geq 2 $ and $[G_{e_-}:t_eG_et_e^{-1}]\geq 2 $ (which is case (3)) because the graph of groups is reduced. 

We now turn to the proof of Theorem \ref{L2 Betti}. We recall the theorem (and slightly rephrase one of the items):

\begin{Thm}
  Let $X$ be a graph, $T$ a maximal tree in $X$, $({\cal G}, X)$ a reduced graph of groups, and $G = \pi_1({\cal G}, X, T)$. Assume that $\beta^1(G_v)=0$ for every $v\in V$ and $\sum\frac{1}{|G_e|}<\8$. Then $\beta^1(G)=0 $ if and only if $G$ belongs to one of the following cases:
  
\begin{enumerate}
 \item The graph $X$ is a single vertex. Then $G= G_v$.
 \item The graph $X$ is a single loop and $G=\Z\ltimes G_v$.  
 \item The graph $X$ is a single edge with $|G_e|<\8$, $e_+\neq e_-$ and $[G_{e_\pm}:G_e]=2$.
  \item Every edge group is infinite.
\end{enumerate}

\end{Thm}

\noindent
\begin{proof}
If $X=\{v\}$ is a single vertex then $G= G_v$ so $\beta^1(G) = 0$.

Assume $X$ is a single loop with $G = \Z\ltimes G_v$. Then, $G_e \cong G_v$ and we have that $\beta^1(G) = 0$ by Proposition \ref{L2Betti EQ}.

If $X$ is a single edge then  $G = G_{e_+}*_{G_e} G_{e_-}$  is an amalgamated product. Assuming that  $|G_e|<\8$ and $[G_{e_\pm}:G_e]=2$. Then we may again apply Proposition \ref{L2Betti EQ} to deduce that 

$$\beta^1(G) = \frac{1}{|G_e|} -  \frac{1}{|G_{e_+}|}-  \frac{1}{|G_{e_+}|}=0.$$

Finally, suppose that every edge group is infinite. It then follows that all vertex groups are infinite as well, and hence by Proposition \ref{L2Betti EQ}, we conclude that $\beta^1(G) = 0$. 

Conversely, suppose that $\beta^1(G)=0$. 






We may assume that we are not in case (1), i.e. $X$ has at least one edge. If $G$ is non-amenable, then by Proposition \ref{VanishH^1} all edge groups are infinite, i.e. we are in case (4). So assume $G$ is amenable, and let $G$ act without inversion on the universal cover $\cal{T}$ of the graph of groups $({\cal G},X)$. By the main result of \cite{PaysValette}, the action of $G$ on ${\cal T}$ is elementary. If $G$ fixes a vertex $v$, then $X=\{v\}$ as $X$ is reduced, and we excluded this. If $G$ fixes two boundary points of ${\cal T}$, then by lemma 18 of \cite{deCornulier}, either $X$ is a loop and $G$ is a semi-direct product $G=\Z\ltimes G_v$ (and we are in case (2)), or $X$ is a segment and $G$ is an amalgamated product with both indices $[G_{e_\pm}:G_e]$ being equal to 2 (and we are in case (3)). If $G$ fixes exactly one boundary point of ${\cal T}$, then by lemma 17 of \cite{deCornulier}, $X$ is a loop and $G$ is an ascending HNN-extension $G=HNN(G_v,\theta)$, where $\theta:G_v\rightarrow G_v$ is injective but not surjective. This of course implies that $G_v$ is infinite. Since $G_e\simeq G_v$, we are in case (4).

\end{proof}

\section{Large groups of automorphisms of $\mathcal{T}$}\label{SectLargeGroups}

In this section we are concerned with closed subgroups $G$ of the automorphism group of a locally finite tree $\mathcal{T}$, acting transitively on the boundary $\partial\mathcal{T}$. It is  known (see Proposition I.10.2 in \cite{FTN}) that $G$ has one or two orbits on the set $V$ of vertices of $\mathcal{T}$, so that $\mathcal{T}$ is either regular or bi-regular. We denote by $\hat{G}$ the dual of $G$, i.e. the set of irreducible unitary  representations of $G$, up to unitary equivalence. 

Pointwise stabilizers in $G$ of finite subtrees of $\mathcal{T}$, form a basis of compact open neighborhoods of the identity in $G$; for $J$ a finite subtree, let $G_J$ be its pointwise stabilizer in $G$. For $\pi\in\hat{G}$, let $P_{\pi,J}$ be the orthogonal projection from the Hilbert space of $\pi$, onto the subspace of $\pi(G_J)$-fixed vectors. We denote by $\ell_\pi$ the minimum cardinality of (the vertex set of) a finite subtree $J$ such that $P_{\pi,J}\neq 0$. Following \cite{FTN}, we say that:
\begin{itemize}
\item $\pi$ is {\it spherical} if $\ell_\pi=1$;
\item $\pi$ is {\it special} if $\ell_\pi=2$;
\item $\pi$ is {\it cuspidal} if $\ell_\pi >2$.
\end{itemize}
Note that $\pi$ is spherical if and only if $\pi$ is a spherical representation with respect to the Gelfand pair $(G,G_a)$, where $G_a$ is the stabilizer of an arbitrary vertex $a\in V$.

Our aim in this section is to give a new proof of a result of Nebbia \cite{Neb} describing $H^1(G,\pi)$, for $\pi\in\hat{G}$; a feature of our proof is that Nebbia appeals to Delorme's theorem \cite{Del} for the vanishing of the first cohomology of a non-trivial spherical representation associated with an arbitrary Gelfand pair. In our  situation, we bypass the use of Delorme's result thanks to the concrete description of spherical representations from \cite{FTN}.

\subsection{The case of two orbits on $V$}

If $G$ has two orbits on $V$, then $G$ acts without inversion on $\mathcal{T}$, with fundamental domain an edge $e=[a,b]$, so $G$ appears as an amalgamated product $G=G_a*_{G_e}G_b$. (Examples are provided by $G=PSL_2(F)$, where $F$ is a non-archimedean local field; or by $G=Aut^+(\mathcal{T})$, the subgroup generated by elliptic automorphisms.) In this case $G$ has a unique special representation $\sigma$ (see Theorem III.2.6 in \cite{FTN}, and the comments following the proof).

We now turn to the proof of Theorem \ref{twoorbits}:

\begin{proof} Let $M_\pi$ be the Hilbert space of $\pi$. If $\ell_\pi>2$, then $M_\pi^{G_e}=\{0\}$, so the result follows from Theorem \ref{main}.

Assume $\ell_\pi=1$; if $\pi$ is the trivial representation, then $H^1(G,\pi)=0$, as $G$ is generated by the union of two compact subgroups (so every homomorphism $G\rightarrow\mathbb{C}$ is trivial). So we may assume that $\pi$ is non-trivial, and appeal to the realization of $\pi$ as a boundary representation, as in Chapter II of \cite{FTN}: the space $M_\pi$ is then a suitable completion of the space of locally constant functions on $\partial\mathcal{T}$, and there exists $s\in ]0,1[\cup(\frac{1}{2}+i\mathbb{R})$ such that the $G$-action is given by
$$\pi(g)\xi(\omega)=P(g,\omega)^s\xi(g^{-1}\omega)$$
where $P(g,\omega)$ is the Radon-Nikodym derivative $\frac{d\nu_{ga}}{d\nu_a}(\omega)$, where $\nu_x$, for $x\in V$, is the unique $G_x$-invariant probability measure on $\partial\mathcal{T}$.

Let $\partial\mathcal{T}=\partial\mathcal{T}_a\cup\partial\mathcal{T}_b$ be the partition of $\partial\mathcal{T}$ induced by the edge $e$: so $\partial\mathcal{T}_a$ is the set of ends $\omega$ such that the ray $[a,\omega[$ does not contain $b$, and vice-versa. Then $M_\pi^{G_a}$ is 1-dimensional (it consists of constant functions on $\partial\mathcal{T}$), $M_\pi^{G_e}$ is 2-dimensional (it consists of functions constant on $\partial\mathcal{T}_a$ and $\partial\mathcal{T}_b$), and $M_\pi^{G_b}$ is 1-dimensional: the latter consists of functions $\xi$ constant on $\partial\mathcal{T}_a$ and $\partial\mathcal{T}_b$, which moreover satisfy:
$$\xi|_{\partial\mathcal{T}_b}=q^{2s}\xi|_{\partial\mathcal{T}_a}$$
where $q+1$ is the degree of the vertex $a$; this follows from the computation of $\frac{d\nu_y}{d\nu_x}$ in Section II.1 of \cite{FTN}. It is then clear that our map $\iota:M_\pi^{G_a}\oplus M_\pi^{G_b}\rightarrow M_\pi^{G_e}$ is onto, so by Theorem \ref{main} we have $H^1(G,\pi)=0$.

Finally we deal with the special representation $\sigma$. Then $M_\sigma^{G_a}=M_\sigma^{G_b}=\{0\}$ (since $\ell_\sigma=2$), so by Theorem \ref{main} we have: $H^1(G,\sigma)\simeq M_\sigma^{G_e}$. By Proposition III.2.3 of \cite{FTN}\footnote{Strictly speaking, this deals with groups having one orbit on vertices, but the comments following Theorem III.2.6 in \cite{FTN} show how to modify it for two orbits.} we have $\dim M_\sigma^{G_e}=1$, completing the proof.
\end{proof}

\subsection{The case of one orbit on $V$}

In this case $\mathcal{T}$ is a $(q+1)$-regular tree, on which $G$ acts with inversions and transitively on the vertex set of $\mathcal{T}$. Examples of this situation are provided by $G=Aut(\mathcal{T})$, or $G=PGL_2(F)$, with $F$ a non-archimedean local field; less classical examples appear in \cite{AmPhD}.

Since $G$ acts with inversions on $\mathcal{T}$, Theorem \ref{main} does not apply immediately. To remedy this, we pass to the first barycentric subdivision $\mathcal{T}_1$ of $\mathcal{T}$, where the assumptions of Theorem \ref{main} hold. 

The new action of $G$ on $\mathcal{T}_1$ has a single edge as a quotient with vertex set $\{a,b\}$ and edge set $\{e\}$. Say that $a$ corresponds to some vertex $\tilde{a}$ of $\mathcal{T}$, and $b$ corresponds to some edge $\tilde{e}$ of $\mathcal{T}$, with $\tilde{a}\in\tilde{e}$. 

With this notation, we have that $G_a = \mathrm{stab}_G(\tilde a)$ and $G_b = \mathrm{stab}_G(\tilde e)$ and $G_e = G_a\cap G_b$. (Here, $\mathrm{stab}_G(\tilde e)$ denotes those elements of $G$ which preserve $\tilde e$ as a set whereas $G_e$ corresponds to the point-wise stabilizer of $\tilde e$ in $G$, so $[\mathrm{stab}_G(\tilde e) :G_e] \leq 2$.) And so
$$G=G_a*_{G_e} G_b.$$

In this case, the pair $(G,G_b)$ is a Gelfand pair (see Lemma II.4.1 in \cite{FTN}). Moreover, up to unitary eqivalence, $G$ has two special representations $\sigma^+,\,\sigma^-$, distinguished by the fact that $\sigma^+$ is a spherical representation for the Gelfand pair $(G,G_b)$, while $\sigma^-$ is not (see Theorem III.2.6 in \cite{FTN}).

The following result has been obtained by Amann \cite{AmMaster} for $G=Aut(\mathcal{T})$, and by Nebbia \cite{Neb} in the general case.

\begin{Thm} Let $G$ be a closed subgroup of $Aut(\mathcal{T})$, acting transitively on $\partial\mathcal{T}$ and $V$. If $\pi\in\hat{G}\backslash\{\sigma^-\}$, then $H^1(G,\pi)=0$; on the other hand $H^1(G,\sigma^-)$ is 1-dimensional.
\end{Thm}

\begin{proof} If $\ell_\pi>2$ or if $\pi$ is the trivial 1-dimensional representation, the proof is the same as for the corresponding cases in Theorem \ref{twoorbits}.

 If $\pi$ is non-trivial and $\ell_\pi=1$, the proof is analogous to the corresponding case in Theorem \ref{twoorbits}: using the realization of $\pi$ as a boundary representation, we have that $M_\pi^{G_a}$ is the 1-dimensional space of constant functions on $\partial\mathcal{T}$, that $M_\pi^{G_e}$ is the 2-dimensional space of functions constant on $\partial\mathcal{T}_a$ and $\partial\mathcal{T}_b$. The only change is that $M_\pi^{G_b}$ is now the 1-dimensional space of functions $\xi$ constant on $\partial\mathcal{T}_a$ and $\partial\mathcal{T}_b$, such that 
 $$\xi|_{\partial\mathcal{T}_b}=q^s\xi|_{\partial\mathcal{T}_a}.$$
 So $\iota:M_\pi^{G_a}\oplus M_\pi^{G_b}\rightarrow M_\pi^{G_e}$ is onto, and the result follows from Theorem \ref{main}.
 
 For $\sigma^+$, we have $M_{\sigma^+}^{G_a}=\{0\}$ (since $\ell_{\sigma^+}=2$), and $M_{\sigma^+}^{G_b}$ is 1-dimensional (since $\sigma^+$ is spherical for $(G,G_b)$), and $M_{\sigma^+}^{G_e}$ is 1-dimensional (by Proposition III.2.3 in \cite{FTN}); so $\iota$ is onto and $H^1(G,\sigma^+)=0$.
 
 Finally, for $\sigma^-$ we have $M_{\sigma^-}^{G_a}=\{0\}$ (since $\ell_{\sigma^-}=2$) and $M_{\sigma^-}^{G_b}=\{0\}$ (as $\sigma^-$ is not spherical for $(G,G_b)$); so $H^1(G,M_{\sigma^-})\simeq M_{\sigma^-}^{G_e}$ by Theorem 1. But $M_{\sigma^-}^{G_e}$ is 1-dimensional, by Proposition III.2.3 in \cite{FTN}.

\end{proof}

\bibliography{bib}

\begin{thebibliography}{BdlHV08}

\bibitem[Ama96]{AmPhD}
Olivier~\'Eric Amann.
\newblock {\em Sur les repr\'esentations unitaires du groupe des automorphismes
  de l'arbre homog\`ene}.
\newblock 1996.
\newblock Thesis (Master)--Univ. de Lausanne.

\bibitem[Ama03]{AmMaster}
Olivier~\'Eric Amann.
\newblock {\em Groups of tree-automorphisms and their unitary representations}.
\newblock 2003.
\newblock Thesis (Ph.D.)--ETH Zurich.

\bibitem[BdlHV08]{BHV}
Bachir Bekka, Pierre de~la Harpe, and Alain Valette.
\newblock {\em Kazhdan's property ({T})}, volume~11 of {\em New Mathematical
  Monographs}.
\newblock Cambridge University Press, Cambridge, 2008.

\bibitem[BH99]{BridsonHaefliger}
Martin~R. Bridson and Andr{\'e} Haefliger.
\newblock {\em Metric spaces of non-positive curvature}, volume 319 of {\em
  Grundlehren der Mathematischen Wissenschaften [Fundamental Principles of
  Mathematical Sciences]}.
\newblock Springer-Verlag, Berlin, 1999.

\bibitem[Bie75]{Bie}
Robert Bieri.
\newblock Mayer-{V}ietoris sequences for {HNN}-groups and homological duality.
\newblock {\em Math. Z.}, 143(2):123--130, 1975.

\bibitem[Bro94]{Brown}
Kenneth~S. Brown.
\newblock {\em Cohomology of groups}, volume~87 of {\em Graduate Texts in
  Mathematics}.
\newblock Springer-Verlag, New York, 1994.
\newblock Corrected reprint of the 1982 original.

\bibitem[CG86]{Cheeger-Gromov}
Jeff Cheeger and Mikhael Gromov.
\newblock {$L_2$}-cohomology and group cohomology.
\newblock {\em Topology}, 25(2):189--215, 1986.

\bibitem[Chi76]{Chi}
I.~M. Chiswell.
\newblock Exact sequences associated with a graph of groups.
\newblock {\em J. Pure Appl. Algebra}, 8(1):63--74, 1976.

\bibitem[dC09]{deCornulier}
Yves de~Cornulier.
\newblock Infinite conjugacy classes in groups acting on trees.
\newblock {\em Groups Geom. Dyn.}, 3(2):267--277, 2009.

\bibitem[Del75]{Del}
Patrick Delorme.
\newblock Sur la {$1$}-cohomologie des repr\'esentations des groupes de {L}ie
  semi-simples. {A}nalyse fonctionnelle.
\newblock {\em C. R. Acad. Sci. Paris S\'er. A-B}, 280:Ai, A1101--A1103, 1975.

\bibitem[FTN91]{FTN}
Alessandro Fig{\`a}-Talamanca and Claudio Nebbia.
\newblock {\em Harmonic analysis and representation theory for groups acting on
  homogeneous trees}, volume 162 of {\em London Mathematical Society Lecture
  Note Series}.
\newblock Cambridge University Press, Cambridge, 1991.

\bibitem[FVM12]{FernosValette}
Talia Fern{\'o}s, Alain Valette, and Florian Martin.
\newblock Reduced 1-cohomology and relative property ({T}).
\newblock {\em Math. Z.}, 270(3-4):613--626, 2012.

\bibitem[Gui72]{Guic}
Alain Guichardet.
\newblock Sur la cohomologie des groupes topologiques. {II}.
\newblock {\em Bull. Sci. Math. (2)}, 96:305--332, 1972.

\bibitem[Hig51]{Hig}
Graham Higman.
\newblock A finitely generated infinite simple group.
\newblock {\em J. London Math. Soc.}, 26:61--64, 1951.

\bibitem[L{\"u}c02]{Leuck}
Wolfgang L{\"u}ck.
\newblock {\em {$L^2$}-invariants: theory and applications to geometry and
  {$K$}-theory}, volume~44 of {\em Ergebnisse der Mathematik und ihrer
  Grenzgebiete. 3. Folge. A Series of Modern Surveys in Mathematics [Results in
  Mathematics and Related Areas. 3rd Series. A Series of Modern Surveys in
  Mathematics]}.
\newblock Springer-Verlag, Berlin, 2002.

\bibitem[MV07]{MV07}
Florian Martin and Alain Valette.
\newblock On the first {$L^p$}-cohomology of discrete groups.
\newblock {\em Groups Geom. Dyn.}, 1(1):81--100, 2007.

\bibitem[MV10]{MarVal}
Florian Martin and Alain Valette.
\newblock Free groups and reduced 1-cohomology of unitary representations.
\newblock In {\em Quanta of maths}, volume~11 of {\em Clay Math. Proc.}, pages
  459--463. Amer. Math. Soc., Providence, RI, 2010.

\bibitem[Neb12]{Neb}
Claudio Nebbia.
\newblock Cohomology for groups of isometries of regular trees.
\newblock {\em Expo. Math.}, 30(1):1--10, 2012.

\bibitem[PV91]{PaysValette}
Isabelle Pays and Alain Valette.
\newblock Sous-groupes libres dans les groupes d'automorphismes d'arbres.
\newblock {\em Enseign. Math. (2)}, 37(1-2):151--174, 1991.

\bibitem[Sch03]{Sch}
James~A. Schafer.
\newblock Graphs of groups and von {N}eumann dimension.
\newblock {\em J. Pure Appl. Algebra}, 180(3):285--297, 2003.

\bibitem[Ser77]{Serre}
Jean-Pierre Serre.
\newblock {\em Arbres, amalgames, {${\rm SL}_{2}$}}.
\newblock Soci\'et\'e Math\'ematique de France, Paris, 1977.
\newblock Avec un sommaire anglais, R{\'e}dig{\'e} avec la collaboration de
  Hyman Bass, Ast{\'e}risque, No. 46.

\bibitem[Sha00]{Shalom1}
Yehuda Shalom.
\newblock Rigidity of commensurators and irreducible lattices.
\newblock {\em Invent. Math.}, 141(1):1--54, 2000.

\bibitem[Swa69]{Swan}
Richard~G. Swan.
\newblock Groups of cohomological dimension one.
\newblock {\em J. Algebra}, 12:585--610, 1969.

\bibitem[Wat82]{Wat}
Yasuo Watatani.
\newblock Property {T} of {K}azhdan implies property {FA} of {S}erre.
\newblock {\em Math. Japon.}, 27(1):97--103, 1982.

\end{thebibliography}
\bibliographystyle{alpha}

\noindent
Authors' address:
\medskip

\noindent
Institut de Math\'ematiques\hfill University of North Carolina at Greensboro\\
Rue Emile Argand 11 \hfill 317 College Avenue\\
CH-2000 Neuch\^atel \hfill Greensboro, NC 27412\\
SWITZERLAND \hfill USA\\
alain.valette@unine.ch \hfill t\_fernos@uncg.edu

\end{document}